\documentclass[a4paper,11pt]{amsart}
\UseRawInputEncoding
\usepackage{amsmath,inputenc,euscript,amssymb,geometry}
\usepackage{graphicx}
\usepackage{advdate}
\usepackage[normalem]{ ulem }
\usepackage{soul}
\usepackage{enumerate}
\usepackage{enumitem}
\usepackage{amssymb}
\usepackage{latexsym}
\usepackage{amssymb,amsbsy,amsmath,amsfonts,amssymb,amscd}
\usepackage{hyperref}
\usepackage{xcolor}

\newtheorem{lemma}{Lemma}[section]
\newtheorem{theorem}[lemma]{Theorem}

\begin{document}
\title[Energy density growth for the Schr\"odinger map and the binormal flow]{Unbounded growth of the energy density associated to the Schr\"odinger map and the binormal flow}
\author[V. Banica]{Valeria Banica}
    \address[V. Banica]{Sorbonne Universit\'e, CNRS, Universit\'e de Paris, Laboratoire Jacques-Louis Lions (LJLL), F-75005 Paris, France, and Institut Universitaire de France (IUF)} 
\email{Valeria.Banica@ljll.math.upmc.fr}
\author[L. Vega]{Luis Vega}
\address[L. Vega]{BCAM-UPV/EHU Bilbao, Spain, luis.vega@ehu.es} 
\email{lvega@bcamath.org}

\maketitle

\begin{abstract}
We consider the binormal flow equation, which is a model for the dynamics of vortex filaments in Euler equations. Geometrically it is a flow of curves in three dimensions, explicitly connected to the 1-D Schr\"odinger map with values on the 2-D sphere, and to the 1-D cubic Schr\"odinger equation. Although these equations are completely integrable we show the existence of an unbounded growth of the energy density. The density is given by the amplitude of the high frequencies of the derivative of the tangent vectors of the curves, thus giving information of the oscillation at small scales. In the setting of vortex filaments the variation of the tangent vectors is related to the derivative of the direction of the vorticity, that according to the Constantin-Fefferman-Majda criterion plays a relevant role in the possible development of singularities for Euler equations. \end{abstract}

\section{Introduction}


In this paper we show the existence of an unbounded flux of the energy density of the solutions of two related partial differential equations which are completely integrable and have a connection with fluid mechanics. The first equation is the Schr\"odinger map in one dimension with values on the 2-D sphere,  known in the physics literature as the classical continuous Heisenberg chain model  in ferromagnetism:
\begin{equation}\label{SM}
T_t=T\times T_{xx}.
\end{equation} 
Here $t$ will represent time and $x$ the spatial variable. Equation \eqref{SM} can be also written in divergence form. This is due to the fact that \eqref{SM} can be obtained by simple differentiation in the spatial variable from the following second equation on curves $\chi(t,x)$ in $\mathbb R^3$:
\begin{equation}\label{bf}
\chi_t=\chi_x\times\chi_{xx}\qquad \chi_x=T, \, |T|=1. 
\end{equation}
This latter equation, known as the Localized Induction Approximation (LIA), and also as the Vortex Filament Equation (VFE) and as the binormal flow (BF), appears naturally as a formal approximation (see \cite{DaR}, \cite{MuTaUkFu}, \cite{ArHa},\cite{CaTi}), after a renormalization of time, of the location evolution of vortex filaments that move according to Euler equations. This model is conjectured to give the right dynamics of vortex filaments in certain situations (see \cite{JeSe} and the references therein). Through this model, at any given time $t$, the curve $\chi(t, \cdot)$ represents the location of the vortex filament, and the tangent vector to the curve $T=\chi_x$ represents the direction of the vorticity. 
\smallskip

A simple use of the Frenet equations immediately gives that \eqref{bf} can be written as
\begin{equation}\label{kin}
\chi_t=cb,
\end{equation}
which explains the binormal flow name, and that
\begin{equation}\label{inter}
T_x=cn,
\end{equation}
where $n$ is the normal vector, $b$ the binormal vector, and $c$ the curvature. It is also easy to see that $c^2(t,x)dx$ is an energy density that from \eqref{kin} describes the kinetic energy of the filament
and from \eqref{inter} the interaction energy of the chain. More precisely 
\begin{equation}\label{energies}\int|\chi_t(t,x)|^2 dx=\int|T_x(t,x)|^2dx=\int c^2(t,x)dx,\end{equation}
and for smooth solutions these quantities are conserved in time if they are finite. 

For analytical reasons it is much more convenient to use instead of the classical Frenet frame given by the tangent, normal and the binormal vectors,  the one given by parallel frames $(T, e_1, e_2)$ constructed as solutions of
\begin{equation}\label{Parallel}\left\{
\begin{split}T_x &=\qquad\quad \alpha e_1+ \beta e_2,\\
e_{1x} &=-\alpha\, T,\\
e_{2x} &=-\beta\, T.
\end{split}
\right.
\end{equation}
Above $\alpha$ and $\beta$ are real scalars. A further simplification can be made defining the complex vector $N=e_1+i e_2\in\mathbb S^2+i\mathbb S^2$ and the complex scalar 
\begin{equation}\label{u}
u=\alpha+i\beta 
\end{equation}
to obtain
\begin{equation}\label{TNtxa}\left\{
\begin{split}T_x &=\Re(\overline u\, N),\\
N_x &=-u\, T.
\end{split}
\right.
\end{equation}
It was proved in \cite{Ha} that in order the constraint $T_{xt}=T_{tx}$ to hold, $u$ has to solve the one dimensional focusing non-linear Schr\"odinger equation (NLS)
\begin{equation}\label{NLS}
iu_t+u_{xx}+\frac 12(|u|^2-a(t))u=0,
\end{equation} 
with $a(t)$ a real scalar, and the tangent and normal vector have to satisfy the linear system
\begin{equation}\label{TNtxb}\left\{
\begin{split}
T_t &=\Im (\overline{u_x}\,N),\\
N_t &=-iu_x T+\frac i2(|u|^2-a(t)) N.
\end{split}
\right.
\end{equation}
Finally, \eqref{bf} writes
\begin{equation}\label{BFu}\chi_t=\Im(\overline u\, N).\end{equation}
It is a well known fact that equation \eqref{NLS} is completely integrable and belongs to the so called AKNS-ZS hierarchy (\cite{ZS},\cite{ZM},\cite{AKNS}; for \eqref{SM} see \cite{Ta}, \cite{ZT}). The geometric meaning of $u$ is clarified when we write \eqref{u} in polar form
\begin{equation}\label{polar}
u(t,x)=\rho(t,x) e^{i\theta(t,x)}
\end{equation}
as $\rho(t)$ and $\theta_x(t)$ are the curvature and the torsion respectively of the curve $\chi(t)$  (\cite{Ha}).\par
Conversely, given a solution $u$ of \eqref{NLS}, $t_0,x_0\in\mathbb R$ and $\mathcal B$ an orthonormal basis of $\mathbb R^3$, one can construct a solution of \eqref{SM} by imposing $\{T,\Re N,\Im N\}(t_0,x_0)=\mathcal B$ and solving \eqref{TNtxb} for $(t,x_0)$ and then \eqref{Parallel} for $(t,x)$. Then, given a point $P\in\mathbb R^3$ a solution of \eqref{bf} is constructed by imposing $\chi(t_0,x_0)=P$ and solving $\chi_t=T \wedge T_{x}$ for $(t,x_0)$ and then $\chi_x=T$ for $(t,x)$. We shall call Hasimoto's method this way of constructing a solution of \eqref{SM}-\eqref{bf} from a solution of \eqref{NLS}.

\smallskip

For $x$ either in the real line $\Bbb R$ or in the torus $\Bbb T$, the well-posedness theory of the initial value problem associated to \eqref{NLS} was established in the function spaces $L^2(\Bbb R)$ in \cite{Ts} and $L^2(\Bbb T)$ in \cite{Bo}. Observe that \eqref {NLS} is invariant under the scaling $u_\lambda(t,x):=\lambda u(\lambda^2 t, \lambda x)$, and that according to this scaling $L^2(\Bbb R)$ is subcritical. Moreover, among the homogenoeus Sobolev spaces, $\dot H^{-\frac 12}$ is the one invariant with respect to the scaling, thus there is a gap of $1/2$ derivative between $L^2$ and the critical space $\dot H^{-\frac 12}$. Starting in \cite{VaVe} a lot of attention has been devoted to extend the well-posedness theory to function and distribution spaces, not necessarily given by the Sobolev class, to make this gap as small as possible. As observed in \cite{Gr} a good choice is to consider the so called Fourier-Lebesgue spaces that are defined using the $L^p$ norm of the Fourier transform of the solution. Therefore, they are invariant under translation in phase space or, equivalently, under the so-called Galilean symmetries. The Fourier-Lebesgue space of functions with Fourier transform in $L^\infty$ is also invariant with respect to the scaling. 
 Several results about ill-posedness, either in the sense that the map datum-solution is not uniformly continuous, or showing what is known as the norm inflation phenomena has been proved (\cite{KPV},\cite{ChCoTa},\cite{CaKa},\cite{Ki},\cite{Oh}). On the other hand, (local) well-posedness holds for data with Fourier transform in $L^p$ spaces, for all $2<p<+\infty$ (\cite{Gr},\cite{Ch}, \cite{GrH}). This result can be proved using perturbation techniques and a fixed point argument. Making a strong use of the complete integrability the gap to the critical space has been also reduced, even in the quite remarkable case of the non-homogeneous Sobolev class and for global well-posedness as recently proved in \cite{HGKV}  to all the subcritical cases, see also \cite{KiViZh},\cite{KoTa}, and \cite{OhWa} for the global in time result in the Fourier-Lebesgue class.  As a consequence, no possible unbounded flux in the size of the Fourier transform of the solution can happen in this subcritical regime. In this paper we focus our attention in the critical case.
 
 \smallskip

Geometrically, critical regularity for \eqref{bf} means the possibility of having either corners or logarithmic spirals. We will concentrate ourselves in the particular case of corners that implies the existence of jumps for the corresponding tangent vectors. The case of logarithmic spirals has been considered in \cite{GV1},\cite{GV2} \cite{Li1}, and \cite{Li2} and it is poorly understood. Nevertheless, we think that it is a quite relevant question to what extent the results showed in this work can be extended to that of the logarithmic spirals. 

The simplest way of obtaining a corner is to look for selfsimilar solutions of \eqref{bf}. That is to say, solutions that can be written as $\chi(t,x)=\sqrt tG(x/\sqrt t)$ for some well chosen $G$. A simple computation gives that such a curve $G$ has to solve the non-linear ODE
\begin{equation}\label{ODE}
\frac12 G-\frac{y}{2}G'=G'\times G''.
\end{equation} 
In \cite{LD},\cite{LRT}, and \cite{Bu} it is proved that a solution of \eqref{ODE} is characterized by the property that the curvature has to be a constant $c=\alpha$, and the torsion $\tau$ has to be $\tau(y)=y/2$. Thanks to \eqref{polar} this amounts to say that 
$$u_\alpha(t,x)=\frac{\alpha}{\sqrt {t}}e^{i\frac{x^2}{4t}}.
$$
As a consequence, if in \eqref{NLS} we take $a(t)=\frac{|\alpha|^2}{t}$ we observe that $u_{\alpha}$ solves \eqref{NLS} with initial condition
$$u_\alpha(0,x)=\sqrt i\alpha \delta(x).$$
Here $\delta$ is the Dirac-$\delta$ function located at the origin. 

Observe that \eqref{bf} is invariant under rotations. Therefore, it is enough to give the  Frenet frame of the curve given by $G$ at say $x=0$ to construct all the solutions of \eqref{ODE}. Take this frame, $\{T(0), n(0), b(0)\}$, to be the canonical orthonormal basis of $\mathbb R^3$, and call $G_\alpha$ the corresponding solution. It was proved  in \cite{GRV} that if
$$\chi_\alpha(t,x)=\sqrt t G_\alpha(\frac x{\sqrt t}),$$
then $\chi_\alpha$ solves \eqref{bf} for $t>0$ and there exists $\chi_\alpha(0,x)$ such that
$$|\chi_\alpha(t,x)-\chi_\alpha(0,x)|\leq2\alpha\sqrt t, \quad t>0.$$
Moreover, $\chi_\alpha(0,x)$ is given by two half lines joined at a corner at the origin. Calling $\theta_\alpha$ the corresponding interior angle,  it is proved in \cite{GRV} that
\begin{equation}\label{angle0}
\sin\frac{\theta_\alpha} 2=e^{-\frac {\pi}{2}|\alpha|^2}.
\end{equation}

In our previous works we have considered two different scenarios for finding a functional setting that includes this example and such that at least a local well-posedness result can be established for the corresponding initial value problem.

\smallskip

The first scenario is when the polygonal line given by $\chi_\alpha(0,x)$ is perturbed in such a way that the angle remains to be $\theta_\alpha$ but outside of the corner location the curve is smooth and tends to two, possibly different, lines at infinity. To find these solutions we study first \eqref{NLS} with $a(t)=\frac{|\alpha|^2}{t}$, then we use Hasimoto's method for positive times and eventually we deal with the limit curves at $t=0$. Regarding \eqref{NLS} we use the pseudo-conformal transformation. More concretely, we look for $v$ with $u=\mathcal T(v)$ where 
\begin{equation}
\label{pseudo}\mathcal{T} (v)(t,x)=\frac{e^{i\frac{x^2}{4t}}}{\sqrt {t}}\,\overline v(\frac 1t,\frac x{t}).
\end{equation}
Observe that above we also make a  change of variables so that the time interval $(0,1)$ becomes $(1, \infty)$. A standard calculation gives that $v$ has to solve
\begin{equation}\label{vNLS}
iv_t+v_{xx}+\frac 1{2t}(|v|^2-\tilde a(t))v=0, \quad \tilde a(t)=\frac1ta(\frac1t)=|\alpha|^2.
\end{equation} 
Solutions of the above equation \eqref{vNLS} formally preserve the $L^2$ norm
\begin{equation}\label{L2}\int |v|^2\, dx,
\end{equation} 
and 
\begin{equation}\label{E}E(v)(t):=\frac12\int|v_x(t)|^2\,dx-\frac{1}{4t}\int(|v|^2-|\alpha|^2)^2\,dx
\end{equation}
satisfies 
$$\partial_tE(v)(t)=\frac{1}{4t^2}\int(|v|^2-|\alpha|^2)^2\,dx.$$
Finally 
$$\Im \int x\bar v_x v\,dx$$
is also a conserved quantity. We could also consider any of the infinitely many conserved quantities of \eqref{NLS} but observe that from the definition of \eqref{pseudo}  it would be necessary to assume  regularity and decay on $v$ for these quantities to be finite. 

Notice that $v_\alpha(t,x):=\alpha$ is a particular solution of \eqref{vNLS}, and the corresponding binormal flow solution is $\chi_\alpha$. 
In a series of papers, see the introduction of \cite{BV4} for a survey of the results, we prove well posedness and small data modified scattering results for $v-v_\alpha$, $t\geq1$, in some appropriate function spaces such that $E(t)$ given in \eqref{E}  is finite. 

\smallskip

The second scenario was started in \cite{BVAnnPDE} and considers solutions of \eqref{bf} that at time $t=0$  are given by a skew polygonal line $\chi_0(x)$ that tends to two lines when $x\rightarrow\pm\infty$. The corners are all located at integers $j\in \Bbb Z$. We use the Hasimoto's method, and at the level of \eqref{NLS} this problem is related to consider data 
$$\sum_j\alpha_j\delta(x-j)$$
with some appropriate conditions in the size of $\alpha_j$. Following \cite{Kita}\footnote{ The authors acknowledge Tohru Ozawa for having pointed to them this article.} we look for solutions of \eqref{NLS} with $a(t)=M/t$, where $M:=\sum_j|\alpha_j|^2$,  for $t>0$, of the type
\begin{equation}\label{ansatz}
u(t,x)=\sum_jA_j(t)\frac{e^{i\frac{(x-j)^2}{4t}}}{\sqrt{t}},
\end{equation}
with
\begin{equation}\label{Aj}A_j(t)=e^{-i(|\alpha_j|^2-M)\log \sqrt{t}}(\alpha_j+R_j(t)),
\end{equation}
and $R_j(t)$ satisfying decay properties as $t$ goes to zero. The construction is performed for $\{\alpha_j\}$ such that 
\begin{equation}\label{weight}
\sum_j|j|^{2s}||\alpha_j|^2<\infty\quad\text{for}\, s>1/2.
\end{equation}
Using Hasimoto's method we construct 
a solution of the binormal flow for $t>0$ such that we recover at time $t=0$ the curve $\chi_0$, provided that 
we choose $\alpha_j$ in a precise way determined by the curvature and torsion angles of $\chi_0$ at $x=j$. In particular, as \eqref{angle0} for self-similar solutions, we choose
\begin{equation}\label{angle}
\sin\frac{\theta_j} 2=e^{-\frac {\pi}{2}|\alpha_j|^2}.
\end{equation}
The need of weights in \eqref{weight} comes from  integrating  \eqref{TNtxb} because the coefficients of that system involve $u_x$. 
Notice that making the expansion of the square phases in \eqref{ansatz} one immediately computes which is  the pseudo-conformal transformation \eqref{pseudo} of $u$. Indeed, we can write $u=\mathcal T(v)$ with $v$ the $2\pi-$periodic function in the $x-$variable:
$$v(t,x)=\sum_jA_j(\frac 1t) e^{-itj^2+ijx},$$
solution of \eqref{vNLS} with $\tilde a(t)=M$.
Observe that \eqref{L2}, the $L^2$ conservation law of \eqref{vNLS}, and \eqref{Aj} give that for $t>0$
\begin{equation}\label{M}\sum_j|A_j(t)|^2=M.
\end{equation}
Recall that from \eqref{Parallel} 
$$|T_x(t,x)|^2=|u(t,x)|^2=\frac1t\left |v(\frac 1t,\frac x{t})\right|^2.$$
Hence for any $t>0$ the function $|T_x(t,x)|^2$ is a $2\pi t$-periodic function in the $x-$variable and by \eqref{M} the integral on each of the periods is $M$. Nevertheless, this ``conservation law" does not give any information about $\hat T_x$, the Fourier transform of $T_x$. We proved in \cite{BVCPDE} that  \eqref{M} can be also understood as a kind of scattering energy of $\hat T_x$ for the solutions of \eqref{SM} and \eqref{bf} that we constructed in \cite{BVAnnPDE}. More precisely,
if
\begin{equation}\label{energydef}
\Xi(T(t)):=\underset{k\rightarrow\infty}{\lim}\int_k^{k+1}|\widehat{T_x}(t,\xi)|^2d\xi,
\end{equation}
then for $t>0$ we have the following conservation law:
\begin{equation}\label{cons}
\Xi(T(t))=4\pi\sum_j|\alpha_j|^2.
\end{equation}
It was also proved in \cite{BVCPDE} that there is a jump discontinuity of $\Xi(T(t))$ at $t=0$. From \eqref{energydef} we can see $|\widehat{T_x}(t,\xi)|^2d\xi$ as an asymptotic energy density in phase space. 
The main result of this paper is to prove that this energy density can grow in time at specific Fourier modes.

The procedure used in \cite{BVAnnPDE} to construct the solution $u$ with the shape given in \eqref{ansatz} is to solve the infinite dimensional non-homogeneous dynamical system generated by  $A_j(t)$. The choice \eqref{M} kills all the resonant frequencies except those generated by the interaction of any mode $j$ with itself. This interaction is easily absorbed by a logarithmic modification of the phase of $A_j(t)$ that has been already incorporated in \eqref{Aj}. Hence, a fixed point argument can be performed  to solve the system and to obtain the decay properties of $R_j(t)$ mentioned above. As a consequence, there is no possible growth for $A_j(t)$. 

The appearance of the logarithmic correction in the phases mentioned above is analogous to the long-range modified scattering that smooth small solutions of \eqref{NLS} satisfy (\cite{Oz}). This modified scattering is behind the growth results of high Sobolev norms proved in \cite{HPTV} for the scalar cubic NLS on $\mathbb R\times\mathbb T^d$ with $d\geq2$. In that setting, which is a mixture of periodic and continuous variables, the authors prove a {\it loglog} growth in time of the amplitudes of the Fourier modes. The key ingredient for this growth is that, differently to what happens in one dimension, for $d\geq2$ the corresponding infinite dimensional system has a non-trivial resonant subsystem that generates solutions whose high Sobolev norms grow in time - see \cite{CKSTT}, \cite{GK}, \cite{Hani}. All these equations are not integrable. 
At this purpose we recall that growth of Sobolev norms for an integrable equation was proved in the case of the cubic Szeg\H o equation (\cite{GeGR}, see also \cite{GLOR}). 

\smallskip

In Theorem \ref{th} below we obtain a precise logarithmic growth in time for $\hat T_x$, being $T$ a solution of  \eqref{SM} and  the tangent vector of a curve that evolves according to \eqref{bf}. This curve  at time $t=0$ is a polygonal line with just two corners of the same angle that are located at $x=1$ and $x=-1$. 
Recall that $T_x$ represents the variation of the direction of the vorticity that as proved in \cite{CFM} plays a crucial role in the possible formation of singularities of Euler equations. 

\begin{theorem}\label{th}
Consider a polygonal line $\chi_0(x)$ with two corners of angle $\theta$ located at $x\in\{-1,1\}$. Let $\chi(t,x)$ be its evolution by the binormal flow \eqref{bf} as explained above and let $T(t,x)$ be its tangent vector. 

There exist $t_\theta,\tilde t_\theta\in (0,1)$ and $n_\theta\in\mathbb N$ such that for $n\in\mathbb N, n\geq n_\theta$, $t\in(\frac{\tilde t_\theta}{n^2\log^2 n},\frac{t_\theta}{n^2})$  and $\xi$ satisfying either $|\xi- \frac{1}{2\pi t}|\leq  \frac{1}{n}$ or $|\xi+ \frac{1}{2\pi t}|\leq  \frac{1}{n}$ the following growth holds:
\begin{equation}\label{Tgrowthcorners}|\widehat{T_x}( t,\xi)- \,V\,\log n|\leq \frac12|V|\log n,\,
\end{equation}
where $V$ is the non-null vector $i(-\frac 2\pi) \log (\sin \frac {\theta}{2})(T^{-\infty}-2T^0+T^{+\infty})$, and the vectors $T^{-\infty}, T^0, T^{+\infty}$ are the directions of the polygonal line $\chi_0(x)$ on $x<-1, -1<x<1$ and $1<x$ respectively. 

As a consequence, for $t\in(0,\frac{t_\theta}{n_\theta^2})$ there exists $C_\theta>0$ such that
\begin{equation}\label{logt}\sup_{\xi}|\widehat{T_x}(t,\xi)|\geq C_\theta \log t.\end{equation}

Finally for $\xi$ satisfying  $|\xi- \frac{1}{2\pi t}|\geq  \frac{3}{8\pi t}$ and $|\xi+ \frac{1}{2\pi t}|\geq  \frac{3}{8\pi t}$ and $t\in(0,\frac{1}{4\pi })$  we have an upper-bound of $|\widehat{T_x}(t,\xi)|$ depending only on $\theta$.

\end{theorem}

Let us first note that our result concerns the growth of $\|\widehat T_x\|_\infty$, the $L^\infty $ norm of the Fourier transform of $T_x$,  and therefore we are a looking at a critical norm in the class of Fourier-Lebesgue spaces. 

The proof Theorem \ref{th}, which is given in \S \ref{sectth}, is based on the computation of $\hat T_x(\xi)$  using \eqref{TNtxa} with $u=\alpha+i\beta$ satisfying \eqref{ansatz}. This generates a first sum in $j$ with the corresponding $\bar A_j$ and their quadratic phases. Then, we use again \eqref{TNtxa} to integrate by parts, and a second sum appears with some new amplitudes $A_r$ and new quadratic phases. It was observed in \cite{BVCPDE} that a resonance can happen if $j-r\in\Bbb Z$ is properly chosen. It is easy to obtain a logarithmic upper bound  for this resonance. It involves a small set of frequencies $\xi$ which does not prevent \eqref{energydef} to hold. Our purpose in this paper to obtain a lower bound. 


In \S \ref{sectN} we prove the  extension of Theorem \ref{th} to the case of polygonal lines with many corners.
The exhibited logarithmic growth is a hint that the numerical computations given in \cite{DHV2} about the unboundedness of $\|\widehat T_x\|_\infty$, are correct - see equation (8) in \S5 in that paper. In that case the initial condition of \eqref{bf} is a planar regular polygon. The dynamics becomes then periodic also in time and exhibits a Talbot effect, in the sense that at rational $p/q$ multiples of the time period, skew polygons emerge with typically as many sides as $q$, see \cite{JeSm2}. So this (numerical) logarithmic growth also happens at these rational times. A rigorous proof of this fact is a very challenging question that we propose to address in the future.

\smallskip

We shall denote systematically by $C(\|A_j(t)\|_{l^1})$ constants depending only on universal constants and on a finite number of positive powers of $\|A_j(t)\|_{l^1}$.

\section{Proof of Theorem \ref{th}}\label{sectth}
 Let $n\in\mathbb N^*$. First we recall that for $s>\frac 12, 0<\gamma<1$, it was proved in \cite{BVAnnPDE} that equation
\begin{equation}\label{NLSbis}iu_t+u_{xx}+\frac 12(|u|^2-\frac{\sum_j|\alpha_j|^2}{t})u=0,\end{equation}
has an unique local solution for $t\in(0,\mathcal T)$ of type
\begin{equation}\label{ansatzbis}u(t,x)=\sum_{j\in\mathbb Z}e^{-i(|\alpha_j|^2-M)\log \sqrt{t}}(\alpha_j+R_j(t))\frac{e^{i\frac{(x-j)^2}{4t}}}{\sqrt{t}},\end{equation}
with
\begin{equation}\label{decayR}\sup_{0<t<\mathcal T}t^{-\gamma}\|R_j(t)\|_{l^{2,s}}+t\,\|\partial_t R_j(t)\|_{l^{2,s}}<C(\gamma)\|\alpha_j\|_{l^{2,s}}^3.\end{equation}
The time of existence $\mathcal T$ is in terms of  $s,\gamma, \|\alpha_j\|_{l^{2,s}}$.

As explained in the Introduction, the evolution $\chi(t)$ of $\chi_0$ on $(0,\mathcal T)$, is constructed by Hasimoto's method from the solution \eqref{ansatzbis} 
of \eqref{NLSbis} 
with, in view of \eqref{angle},
$$\alpha_j=\sqrt{(-\frac 2\pi) \log (\sin \frac {\theta}{2})}=:\alpha \mbox{ for }j\in\{\pm 1\},\quad \alpha_j=0\mbox{ otherwise}.$$

\subsection{A general analysis on locating possible growth scenarios}

We shall start with a  lemma that highlights the part of $\widehat{T_x}(t,\xi)$ that can grow for small times, in general cases of polygonal lines.  
\begin{lemma}\label{growth} Let $\{\alpha_j\}\in l^{2,\frac 12^+}$ and let $\chi(t)$ be the evolution through the binormal flow of the corresponding polygonal line.  
For $t\in (0,\frac 1{4\pi n^2})$ and $\xi\in\mathbb R$ the tangent vector of $\chi(t)$ satisfies
$$\left|\widehat{T_x}(t,\xi)-i\sum_{|j- r+ 4\pi t\xi|<2nt}\,\overline{A_j(t)}A_r(t)\, e^{-i\frac{j^2-r^2}{4t}}\right.$$
$$\left.\times\int_{|x-j-4\pi t\xi|>\frac 1n,\,|x-r+4\pi t\xi|>\frac 1n}e^{i\frac{x( j- r+ 4\pi t\xi)}{2t}}\left(\frac{1}{x-j-4\pi t\xi}-\frac{1}{x-r+ 4\pi t\xi}\right)\,T(t,x)\,dx\right|\leq \frac{C(\|A_j(t)\|_{l^{1}})}{n\sqrt{t}}.$$

\end{lemma}
\begin{proof}

From \eqref{TNtxa} we have $T_x(t,x)=\Re(\overline{u} N)(t,x)$ so
$$\widehat{T_x}(t,\xi)=\int_{-\infty}^\infty e^{i2\pi x\xi}\,\Re(\overline{u} N)(t,x)dx=\int_{-\infty}^\infty e^{i2\pi x\xi}\,\Re(\sum_j\overline{A_j(t)}\frac{e^{-i\frac{(x-j)^2}{4t}}}{\sqrt{t}} N(t,x))dx$$
$$=\frac{e^{i4\pi^2t\xi^2}}{2\sqrt{t}}\sum_{\pm, j}e^{i2\pi j\xi}\,\overline{A_j(t)}\int_{-\infty}^\infty e^{-i\frac{(x-j-4\pi t\xi)^2}{4t}}N(t,x)dx$$
$$+\frac{e^{-i4\pi^2t\xi^2}}{2\sqrt{t}}\sum_{\pm, j}e^{i2\pi j\xi}\,A_j(t)\int_{-\infty}^\infty e^{i\frac{(x-j+4\pi t\xi)^2}{4t}}\overline{N (t,x)}dx.$$
We start by removing bounded pieces of the integral centered in $j\pm4\pi t\xi$, by using a cut-off function $\psi_n$ vanishing on $B(0,\frac 1{4n})$ and valued $1$ on $^cB(0,\frac 1{2n})$. These pieces are easy to estimate by $C\frac{\|A_j(t)\|_{l^{1}}}{n\sqrt{t}}$, as the integrants are of constant modulus. On the remaining pieces we integrate by parts:
$$\left|\widehat{T_x}(t,\xi)+i\sqrt{t}e^{i4\pi^2t\xi^2}\sum_{\pm, j}e^{i2\pi j\xi}\,\overline{A_j(t)}\int_{-\infty}^\infty e^{-i\frac{(x-j-4\pi t\xi)^2}{4t}}\left(\frac{N(t,x)\psi_n(x-j-4\pi t\xi)}{x-j-4\pi t\xi}\right)_xdx\right.$$
$$\left.-i\sqrt{t}e^{-i4\pi^2t\xi^2}\sum_{\pm, j}e^{i2\pi j\xi}\,A_j(t)\int_{-\infty}^\infty  e^{i\frac{(x-j+4\pi t\xi)^2}{4t}}\left(\frac{\overline{N(t,x)}\psi_n(x-j+4\pi t\xi)}{x-j+4\pi t\xi}\right)_x dx\right|\leq C\frac{\|A_j(t)\|_{l^{1}}}{n\sqrt{t}}.$$
When the derivative falls on $\psi_n$ and on the denominator we get terms bounded by $Cn\sqrt{t}\|\{A_j(t)\}\|_{l^1}$, and $n\sqrt{t}\leq \frac 1{n\sqrt{t}}$ as $t\in (0,\frac 1{4\pi n^2})$. We are left with the part from $N_x=-uT$. We get then
$$\left|\widehat{T_x}(t,\xi)-i\sum_{j}\,\overline{A_j(t)} \int_{-\infty}^\infty \frac{\sum_r A_r(t)e^{i\frac{x(j-r+4\pi t\xi)}{2t}}e^{-i\frac{j^2-r^2}{4t}}}{x-j-4\pi t\xi}T(t,x)\psi_n(x-j-4\pi t\xi)dx\right.$$
$$\left.+i\sum_{j}\,A_j(t)\int_{-\infty}^\infty  \frac{\sum_r \overline{A_r(t)}e^{i\frac{x(r-j+4\pi t\xi)}{2t}}e^{i\frac{j^2-r^2}{4t}}}{x-j+4\pi t\xi}T(t,x)\psi_n(x-j+4\pi t\xi)dx\right|\leq C\frac{\|A_j(t)\|_{l^{1}}}{n\sqrt{t}}.$$
For $|\frac{\pm(j- r)+ 4\pi t\xi}{2t}|\geq n$ we perform an integration by parts using the linear phase and get terms bounded by $Ct\|\{A_j(t)\}\|_{l^1}^2$ when the derivative falls on $\psi_n$ and on the denominator. When the derivative falls on $T$ we get $\Re(\overline{u}N)$ which yields a quadratic phase. We complete to a quadratic phase incorporating the linear one and we remove bounded pieces localized where the quadratic phases cancel, which are upper-bounded by $C\sqrt{t}\|\{A_j(t)\}\|_{l^1}^3$. Then we perform an integration by parts from the quadratic phase and get terms upper-bounded by $Cnt\sqrt{t}\|\{A_j(t)\}\|_{l^1}^3+Cnt\|\{A_j(t)\}\|_{l^1}^4$. 
Summarizing we have obtained for all $\xi\in\mathbb R$
$$\left|\widehat{T_x}(t,\xi)-i\sum_{|j- r+ 4\pi t\xi|<2nt}\,\overline{A_j(t)}A_r(t)\, I^+(t,\xi,j,r)+i\sum_{|-( j- r)+ 4\pi t\xi|<2nt}\,A_j(t)\overline{A_r(t)}\, I^-(t,\xi,j,r)\right|$$
$$\leq \frac{C(\|A_j(t)\|_{l^{1}})}{n\sqrt{t}},$$
where
$$I^\pm(t,\xi,j,r):=e^{\mp i\frac{j^2-r^2}{4t}}\int_{|x-j\mp 4\pi t\xi|>\frac 1n} \frac{e^{i\frac{x(\pm( j- r)+ 4\pi t\xi)}{2t}}}{x-j\mp 4\pi t\xi}\,T(t,x)\,dx.$$
The first discrete summation holds for $(j,r)$ if and only if the second discrete summation holds for $(r,j)$, so
$$\left|\widehat{T_x}(t,\xi)-i\sum_{|j- r+ 4\pi t\xi|<2nt}\,\overline{A_j(t)}A_r(t)\, (I^+(t,\xi,j,r)- I^-(t,\xi,r,j))\right|\leq \frac{C(\|A_j(t)\|_{l^{1}})}{n\sqrt{t}}.$$
We then have
$$\left|\widehat{T_x}(t,\xi)-i\sum_{|j- r+ 4\pi t\xi|<2nt}\,\overline{A_j(t)}A_r(t)\, e^{-i\frac{j^2-r^2}{4t}}\right.$$
$$\left.\times\int_{|x-j-4\pi t\xi|>\frac 1n,\,|x-r+4\pi t\xi|>\frac 1n}e^{i\frac{x( j- r+ 4\pi t\xi)}{2t}}\left(\frac{1}{x-j-4\pi t\xi}-\frac{1}{x-r+ 4\pi t\xi}\right)\,T(t,x)\,dx\right|$$
$$\leq \frac{C(\|A_j(t)\|_{l^{1}})}{n\sqrt{t}}.$$
\end{proof}

\subsection{An analysis on locating particular solutions that can exhibit growth.}

To get the logarithmic growth of the theorem we shall restrict to a particular class of polygonal lines and we shall look for values of $\xi_n$ such that $|4\pi \frac t{n^2}\xi_n|$ is close to the distance between the corners. We note that in the case of a single corner we have $\alpha_1\alpha_{-1}=0 $ so the following Lemma ensures us that $\widehat{T_x}(\frac t{n^2},\xi_n)$ is bounded. Therefore the logarithmic growth comes from the interaction of several corners.

\begin{lemma}\label{corgrowth}Let $\{\alpha_j\}$ such that 
\begin{equation}\label{hyp}
\alpha_j=0 \mbox{ for }|j|>1.
\end{equation}
We have for all $t\in(0,\frac 1{4\pi n^2})$ and $|\delta|<\frac 1{n}$:
\begin{equation}\label{T^nearZ/tbis}
\left|\widehat{T_x}(t,\frac{1}{2\pi t}+\delta)-i\overline{\alpha_{-1}}\alpha_{1}\, \int_{|x-1|>\frac 1n,\,|x+1|>\frac 1n,\,|x|<2}\left(\frac{1}{x-1}-\frac{1}{x+1}\right)\,T(t,x)\,dx\right|
\end{equation}
$$\leq \frac{C(\|A_j(t)\|_{l^{1}})}{n\sqrt{t}}.$$
A similar estimate holds at $-\frac{1}{2\pi t}+\delta$.
\end{lemma}

\begin{proof}
Denoting $\xi=\frac{1}{2\pi t}+\delta$ implies $4\pi t\xi=2+4\pi t\delta$ so we have
$$|j- r+ 4\pi t \xi|<2nt<\frac 1{2\pi n}\iff r=j+2,$$
and Lemma \ref{growth} gives us
\begin{equation}\label{fullint}\left|\widehat{T_x}(t,\xi)-i\sum_{j}\,\overline{A_j(t)}A_{j+2}(t)\, e^{i\frac{(j+2)^2-j^2}{4t}}\right.\end{equation}
$$\left.\times\int_{|x-j-2-4\pi t\delta|>\frac 1n,\,|x-j+4\pi t\delta|>\frac 1n}e^{ix 2\pi \delta}\left(\frac{1}{x-j-2-4\pi t\delta}-\frac{1}{x-j+ 4\pi t\delta}\right)\,T(t,x)\,dx\right|$$
$$\leq \frac{C(\|A_j(t)\|_{l^{1}})}{n\sqrt{t}}.$$
First we note that, as $|4\pi t\delta|<\frac 1n$, we obtain by Cauchy-Schwarz
$$\int_{|x-j-2-4\pi t\delta|>\frac 1n,\,|x-j+4\pi t\delta|>\frac 1n,\,|x-j-1|>2}\left|\frac{1}{x-j-2-4\pi t\delta}-\frac{1}{x-j+ 4\pi t\delta}\right|\,dx$$
$$=\int_{|y-1-4\pi t\delta|>\frac 1n,\,|y+1+4\pi t\delta|>\frac 1n,\,|y|>2}\left|\frac{2+8\pi t\delta}{(y-1-4\pi t\delta)(y+1+ 4\pi t\delta)}\right|\,dy\leq C$$
Thus we can reduce the integration in \eqref{fullint} to $|x-j-1|<2$. The remaining integrals are upper-bounded by $\log n$, and as $|x2\pi \delta|\in (0,\frac {4\pi}n)$ and $|4\pi t\delta|<\frac 1n$ we obtain
 $$\left|\widehat{T_x}(t,\xi)-i\sum_{j}\,\overline{A_j(t)}A_{j+2}(t)\, e^{i\frac{(j+2)^2-j^2}{4t}}\,\right.$$
 $$\times\left.\int_{|x-j-2|>\frac 1n,\,|x-j|>\frac 1n,\,|x-j-1|<2}\left(\frac{1}{x-j-2}-\frac{1}{x-j}\right)\,T(t,x)\,dx\right|\leq \frac{C(\|A_j(t)\|_{l^{1}})}{n\sqrt{t}}.$$
Now using the fact that $\alpha_j=0 \mbox{ for }|j|>1$ and again that the integrals are upper-bounded by $\log n$ we have
$$\left|\widehat{T_x}(t,\xi)-i\overline{\alpha_{-1}}\alpha_{1}\, \int_{|x-1|>\frac 1n,\,|x+1|>\frac 1n,\,|x|<2}\left(\frac{1}{x-1}-\frac{1}{x+1}\right)\,T(t,x)\,dx\right|\leq \frac{C(\|A_j(t)\|_{l^{1}})}{n\sqrt{t}}$$
$$+\sum_{|j|\leq 1} |\alpha_j||R_{j+2}(t)|\log n+\sum_{|j+2|\leq n} |R_j(t)||\alpha_{j+2}|\log n+\sum_j |R_j(t)||R_{j+2}(t)|\log n.$$
By using the decay \eqref{decayR} of $R_j(t)$ with $\gamma >\frac 12$, we get 
\begin{equation}\label{decayRbis}|R_j(t)|\leq \|R_k(t)\|_{l^{2,s}}\leq  C(\gamma) t^\gamma \|\alpha_k\|_{l^{2,s}}^3\leq \frac {C}{n}\|\alpha_k\|_{l^{2,s}}^3,\end{equation} 
thus the conclusion \eqref{T^nearZ/tbis} of the Lemma. 
For $\xi=-\frac{1}{2\pi t}+d$ we proceed the same.\\

\end{proof}

\subsection{Proof of the logarithmic growth}\label{subsectth}

 Lemma \ref{corgrowth} ensures us that for $\xi_n=\frac{n^2}{2\pi t}+\delta$ with $t\in(0,\frac 1{4\pi})$ and $|\delta|<\frac 1n$ we have:
 $$\left|\widehat{T_x}(\frac t{n^2},\xi_n)-i|\alpha|^2\, \int_{|x-1|>\frac 1n,\,|x+1|>\frac 1n,\,|x|<2}\left(\frac{1}{x-1}-\frac{1}{x+1}\right)\,T(\frac t{n^2},x)\,dx\right|\leq \frac{C(\|A_j(\frac t{n^2})\|_{l^{1}})}{\sqrt{t}}.$$
Now we note that we can further reduce to
 $$\left|\widehat{T_x}(\frac t{n^2},\xi_n)-i|\alpha|^2\, \int_{\frac 13>|x-1|>\frac 1n,\,\frac 13>|x+1|>\frac 1n}\left(\frac{1}{x-1}-\frac{1}{x+1}\right)\,T^n(\frac t{n^2},x)\,dx\right|\leq \frac{C(\|A_j(\frac t{n^2})\|_{l^{1}})}{\sqrt{t}}.$$
Therefore
 $$\left|\widehat{T_x}(\frac t{n^2},\xi_n)-i|\alpha|^2\, \int_{\frac 13>|x-1|>\frac 1n,\,\frac 13>|x+1|>\frac 1n}\left(\frac{1}{x-1}-\frac{1}{x+1}\right)\,T\left(\frac{t}{n^2},x\right)\,dx\right|\leq \frac{C(\|A_j(\frac t{n^2})\|_{l^{1}})}{\sqrt{t}}.$$
We recall also that from Lemma 4.1 in \cite{BVAnnPDE} we have the convergence
$$|T(\frac t{n^2},x)-T(0,x)|\leq C(\|\alpha_j\|_{l^{1,1}})(1+|x|)\sqrt{\frac t{n^2}}\left(\frac 1{d(x,\frac 12\mathbb Z)}+\frac 1{d(x,\mathbb Z)}\right),$$
so
 $$\left|\widehat{T_x}(\frac t{n^2},\xi_n)-i|\alpha|^2\, \int_{\frac 13>|x-1|>\frac 1n,\,\frac 13>|x+1|>\frac 1n}\left(\frac{1}{x-1}-\frac{1}{x+1}\right)\,T\left(0,x\right)\,dx\right|$$
$$\leq \frac{C(\|A_j(\frac t{n^2})\|_{l^{1}})}{\sqrt{t}}+C|\alpha|^3\sqrt{t}\log n.$$
As $\chi_0$ is a polygonal line with $T^{-\infty}, T^0$ and $T^{+\infty}$ the directions on $x<-1, -1<x<1$ and $1<x$ respectively, 
$$\int_{\frac 13>|x-1|>\frac 1n,\,\frac 13>|x+1|>\frac 1n}\left(\frac{1}{x-1}-\frac{1}{x+1}\right)\,T\left(0,x\right)\,dx$$
$$=T^{-\infty}\left[\log\frac{|x-1|}{|x+1|}\right]_{-\frac 43 }^{-1-\frac 1n}+T^0\left[\log\frac{|x-1|}{|x+1|}\right]_{-1-\frac 1n}^{-\frac 23 }+T^0\left[\log\frac{|x-1|}{|x+1|}\right]_{\frac 23 }^{1-\frac 1n}+ T^{+\infty}\left[\log\frac{|x-1|}{|x+1|}\right]_{1+\frac 1n}^{\frac 43 },$$
and a logarithmic growth in $n$ comes from the boundary terms at $-1-\frac 1n,-1+\frac 1n,1-\frac 1n$ and $1+\frac 1n$. 
Therefore we obtain
$$\left|\widehat{T_x}(\frac t{n^2},\xi_n)- i|\alpha|^2(T^{-\infty}-2T^0+T^{+\infty})\log n\right|\leq \frac{C(\|A_j(\frac t{n^2})\|_{l^{1}})}{\sqrt{t}}+C|\alpha|^3\sqrt{t}\log n.$$
Now we compute:
\begin{equation}\label{smallnessvector}
|T^{-\infty}-2T^0+T^{+\infty}|=2(1-\cos(\pi-\theta)).
\end{equation}
The smallest values of this modulus appears for $\theta$ close to $\pi$, equivalent to $\alpha$ close to zero im view of \eqref{angle}:
$$1-\frac{(\pi-\theta)^2}8\overset{\theta\rightarrow \pi}{\approx}\cos\frac{\pi-\theta}2=\sin\frac \theta 2=e^{-\frac\pi 2|\alpha|^2}\overset{\alpha\rightarrow 0}{\approx}1-\frac\pi 2|\alpha|^2,$$
thus
$$|T^{-\infty}-2T^0+T^{+\infty}|=2(1-\cos(\pi-\theta))\overset{\theta\rightarrow \pi}{\approx}(\pi-\theta)^2\approx 4\pi|\alpha|^2.$$
We then choose $t$ small enough such that
$$C|\alpha|^3\sqrt{t}<\frac{1}8|\alpha|^2|T^{-\infty}-2T^0+T^{+\infty}|,$$
that reduces to 
$$\sqrt t<C|\alpha|,$$
and large enough such that
$$\frac{C(\|A_j(\frac t{n^2})\|_{l^{1}})}{\sqrt{t}}<\frac{1}8|\alpha|^2|T^{-\infty}-2T^0+T^{+\infty}|\log n,$$
that is implied, in view of \eqref{decayRbis} by 
$$\frac{C(|\alpha|+\frac{|\alpha|^3}{n})}{|\alpha|^4\log n}<\sqrt t.$$
Therefore the conditions on $t$ and $n$ with respect to $\alpha$ are
\begin{equation}\label{condtn}
\frac{C(|\alpha|+\frac{|\alpha|^3}{n})}{|\alpha|^4\log n}<\sqrt t<\min\{C|\alpha|,\frac 1{4\pi}\}.
\end{equation}
We note that the upper and lower condition on $t$ imply that $n$ has to be chosen large with respect to $\alpha$. \\

Summarizing we have obtained the existence of  $t_\theta,\tilde t_\theta\in (0,1)$ and $n_\theta\in\mathbb N$ such that for $n\in\mathbb N, n\geq n_\theta$ and $t\in(\frac{\tilde t_\theta}{\log^2 n},t_\theta)$ the following growth holds
$$|\widehat{T_x}(\frac t{n^2},\frac{n^2}{2\pi t}+\delta) -\, i|\alpha|^2(T^{\-\infty}-2T^0+T^{+\infty})\log n)|\leq \frac12|\alpha|^2|T^{-\infty}-2T^0+T^{+\infty}|\log (n),\,\forall \delta, |\delta|<\frac1{n}.$$
This yields  \eqref{Tgrowthcorners} in Theorem \ref{th}. For the analysis at $-\frac{n^2}{2\pi t}+\delta$ we proceed the same way. \\

As a consequence we get the existence of $C_\theta>0$ such that for $n\geq n_\theta$ and $t\in(\frac{\tilde t_\theta}{n^2},\frac{t_\theta}{n^2})$ we have
$$\sup_{\xi}|\widehat{T_x}(t,\xi)|\geq C_\theta \log n.$$
We can choose $n_\theta$ large enough such that for all $\tau\in (\tilde t_\theta,t_\theta)$ 
$$|\log \tau|<\frac 12\left |\log \frac{ t_\theta}{n_\theta^2}\right|<\frac 12\left |\log \frac{\tau}{n_\theta^2}\right|.$$
Then, for $t=\frac{\tau}{n^2}\in(\frac{\tilde t_\theta}{n^2},\frac{t_\theta}{n^2})$
$$\sup_{\xi}|\widehat{T_x}(t,\xi)|\geq \frac{C_\theta}2( |\log t|-|\log\tau|) \geq \frac{C_\theta}4 \log t.$$
By choosing moreover $n_\theta$ large enough such that 
$$\frac{t_\theta}{(n+1)^2}>\frac{\tilde t_\theta}{n^2},$$
we have that $(0,\frac{t_\theta}{n_\theta^2})=\cup_{n\geq n_\theta}(\frac{\tilde t_\theta}{n^2},\frac{t_\theta}{n^2})$ so \eqref{logt} holds.

\subsection{Bounds away from the growth zone}\label{sectbound}

Finally we consider for $n\in\mathbb N^*$  those  $\xi$ such that $|\xi-\frac{n^2}{2\pi t}|\geq \frac{3n^2}{2\pi t}$ and $|\xi+ \frac{n^2}{2\pi t}|\geq \frac{3n^2}{2\pi t}$, which means $|4\pi \frac t{n^2}\xi-2|\geq \frac 34$ and $|4\pi \frac t{n^2}\xi+2|\geq \frac 34$. We denote $m\in\frac{\mathbb Z}2$ and $d\in[-\frac 12,\frac 12[$ the numbers such that $4\pi \frac t{n^2}\xi=2m+d$, that in particular implies $m\notin\{\pm 1\}$. Then 
$$|j- r+ 4\pi \frac t{n^2}\xi|<2n\frac t{n^2}\iff r=j+2m,$$
and Lemma \ref{growth} gives us for $t\in (0,\frac1{4\pi })$
$$|\widehat{T_x}(\frac t{n^2},\xi)|\leq \sum_{j}|A_j(\frac t{n^2})||A_{j+2m}(\frac t{n^2})|\int_{ |x-j-2m-d|>\frac 1n,\,|x-j+d|>\frac 1n}\left|\frac{1}{x-j-2m-d}-\frac{1}{x-j+d}\right|\,dx$$
$$+\frac{C(\|A_j(\frac t{n^2})\|_{l^{1}})}{\sqrt{t}}.$$
The piece of integration $|x-j-m|>4m$ is bounded and the remaining part is upper-bounded by $C \log (\max\{\langle m\rangle, n\})$. Since $m\notin\{\pm 1\}$ then
$$\sum_{j}|A_j(\frac t{n^2})||A_{j+2m}(\frac t{n^2})|\leq \sum_{|j|\leq 1}|\alpha_j||R_{j+2m}(\frac t{n^2})|+\sum_{|j+2m|\leq 1}|R_j(\frac t{n^2})||\alpha_{j+2m}|+\sum_{j}|R_j(\frac t{n^2})||R_{j+2m}(\frac t{n^2})|$$
$$=|\alpha|(|R_{-1+2m}(\frac t{n^2})|+|R_{1+2m}(\frac t{n^2})|+|R_{-1-2m}(\frac t{n^2})|+|R_{1-2m}(\frac t{n^2})|)+\sum_{j}|R_j(\frac t{n^2})||R_{j+2m}(\frac t{n^2})|.$$
From \eqref{decayR} we have:
$$\sum_j |j|^{2s}|R_j(\frac t{n^2})|^2\leq C\frac{|\alpha|^3}{n},$$
so for all $j\in\mathbb Z$ 
$$ |R_j(\frac t{n^2})| \leq \frac{C|\alpha|^3}{\sqrt{n}\langle j\rangle^s}.$$
Thus
$$|\widehat{T_x}(\frac t{n^2},\xi)|\leq C(\alpha)\log(\max\{\langle m\rangle, n\}) \frac {1}{\sqrt{n}\langle m\rangle^s} +\frac{C(\alpha)}{\sqrt{t}}.$$
Therefore for $\xi$ such that $|\xi-\frac{n^2}{2\pi t}|\geq \frac{3n^2}{2\pi t}$ and $|\xi+ \frac{n^2}{2\pi t}|\geq \frac{3n^2}{2\pi t}$ and $t\in (0,\frac 1{4\pi})$ we obtain
$$|\widehat{T_x}(\frac t{n^2},\xi)|\leq \frac{C(\alpha)}{\sqrt{t}},$$
and the proof of Theorem \ref{th} is completed. 

\section{Several corners}\label{sectN}
Instead of $2$ corners we consider a planar polygonal line with $2N$ corners located at $-2N+1,...,2N-1$, with same angle $\theta$. Let $m\in\{1,...,N\}$. Proceeding similarly as in the proof of Lemma \ref{corgrowth} we obtain (with constants that can depend on $N$) for $|\delta|<\frac 1n$:
$$\left|\widehat{T_x}(\frac t{n^2},\frac{mn^2}{2\pi t}+\delta)-i|\alpha|^2\sum_{j\in\{(-2N+1),..,(2N-2m-1)\}}\, \right.$$
$$\left.\times e^{i\frac{(j+2m)^2-j^2}{4\frac t{n^2}}} \int_{\frac 1n<|x-(j+2m)|<\frac 13,\,\frac 1n<|x-j|<\frac 13}\left(\frac{1}{x-(j+2m)}-\frac{1}{x-j}\right)\,T(t,x)\,dx\right|\leq \frac{C(\|A_j(\frac t{n^2})\|_{l^{1}})}{\sqrt{t}}.$$
Now we restrict to $t$ and $n$ such that $\frac {n^2}{t}\in 8\pi\mathbb Z$ to get rid of the phases in front of the integrals. Arguing as for the end of the proof of Theorem \ref{th} in \S \ref{subsectth}, we get
$$\left|\widehat{T_x}(\frac t{n^2},\frac{mn^2}{2\pi t}+\delta)-V_m\log n\right|\leq \frac{C(\|A_j(t)\|_{l^{1}})}{\sqrt{t}}+C|\alpha|^3\sqrt{t}\log n.$$
where
\begin{equation}\label{Vm}V_m=i(-\frac 2\pi) \log (\sin \frac {\theta}{2})\sum_{j\in\{(-2N+1),..,(2N-2m-1)\}}\, (T(j^-)-T(j^+)-T((j+2m)^-)+T((j+2m)^+))\end{equation}
$$=i(-\frac 2\pi) \log (\sin \frac {\theta}{2})(T((-2N+1)^-)-T((-2N+1+2m)^-) -T((2N-1-2m)^+)+T((2N-1)^+)$$
For $N=1$ we recover the vector $V$ in Theorem \ref{th}, as
$$V_1=i(-\frac 2\pi) \log (\sin \frac {\theta}{2})(T((-2N+1)^-)-T((-2N+1)^+) -T((2N-1)^-)+T((2N-1)^+).$$
We note that for at least one $m\in\{1,...,N\}$ we have $V_m\neq 0$.\\

Continuing similarly as in \S\ref{subsectth}-\ref{sectbound} we obtain the following result.

\begin{theorem}\label{thN}
Consider a polygonal line $\chi_0(x)$ with $2N$ corners of angle $\theta$ located at $x\in\{-2N+1,...,2N-1\}$. Let $\chi(t,x)$ be its evolution by the binormal flow by the Hasimoto method and denote $T(t,x)$ the tangent vector. 

There exists  $t_{\theta,N},\tilde t_{\theta,N}\in (0,1)$ and $n_{\theta,N}\in\mathbb N$ such that for $n\in\mathbb N, n\geq n_{\theta,N}$, for $t\in(\frac{\tilde t_{\theta,N}}{n^2\log^2 n},\frac{t_{\theta,N}}{n^2})$ satisfying $\frac 1t\in 8\pi\mathbb Z$, for $m\in\{1,...,N\}$ and for $\xi$ satisfying either $|\xi- \frac {m}{2\pi t}|\leq \frac 1n$ or $|\xi+ \frac {m}{2\pi t}|\leq \frac 1n$, the following holds
\begin{equation}\label{TgrowthNcorners}|\widehat{T_x}(t,\xi)-V_m\,\log (n) |\leq \frac 12 |V_m|\log n,\end{equation}
where $V_m$ is defined in \eqref{Vm}.

As a consequence, for $t\in(0,\frac{t_{\theta,N}}{n_{\theta,N}^2})$ there exists $C_{\theta,N}>0$ such that
\begin{equation}\label{logtN}\sup_{\xi}|\widehat{T_x}(t,\xi)|\geq C_{\theta,N} \log t.\end{equation}

Finally for $\xi$ such that  $|\xi- \frac{m}{2\pi t}|\geq  \frac{3}{8\pi t}$ and $|\xi+ \frac{m}{2\pi t}|\geq  \frac{3}{8\pi t}$ for all $m\in\{1,...,N\}$,  and $t\in(0,\frac{1}{4\pi })$  we have an upper-bound of $|\widehat{T_x}(t,\xi)|$ depending only on $\theta$ and $N$.

\end{theorem}

{\bf{Acknowledgements:}}  This research is partially supported by the Institut Universitaire
de France, by the French ANR project SingFlows, by ERCEA Advanced Grant 2014 669689
- HADE, by MEIC (Spain) projects Severo Ochoa SEV-2017-0718, and PGC2018-1228
094522-B-I00, and by Eusko Jaurlaritza project IT1247-19 and BERC program.

\end{document}